\documentclass[12pt]{article}
\usepackage[utf8]{inputenc}
\usepackage{amsfonts} 
\usepackage{amsthm}
\usepackage{amssymb}
\usepackage{amsmath}
\usepackage{comment}

\usepackage{color}

\usepackage{graphicx}

\usepackage{yhmath}

\newtheorem{theorem}{Theorem}
\newtheorem{proposition}{Proposition}

\newtheorem{lemma}{Lemma}

\theoremstyle{definition}

\newtheorem{remark}{Remark}

\newcommand{\N}{\mathbb{N}}
\newcommand{\Z}{\mathbb{Z}}

\newcommand{\Pp}{\mathbb{P}}

\title{Anisotropic non-oriented bond percolation in high dimensions}
\author{Pablo A. Gomes\footnote{Universidade de São Paulo, Brasil. pagomes@usp.br} \and Alan Pereira\footnote{Universidade Federal de Alagoas, Brasil. alan.pereira@im.ufal.br} \and Remy Sanchis\footnote{Universidade Federal de Minas Gerais, Brasil. rsanchis@mat.ufmg.br}}
\date{}

\begin{document}

\maketitle

\begin{abstract}
    We consider inhomogeneous non-oriented Bernoulli bond percolation on $\Z^d$, where each edge has a parameter depending on its direction. We prove that, under certain conditions, if the sum of the parameters is strictly greater than 1/2, we have percolation in sufficiently high dimensions. The main tool is a dynamical coupling between  models for different dimensions with different sets of parameters.
    \medskip
    
\smallskip

\noindent \textbf{AMS-subject:}{ 60K35; 82B43} 
\end{abstract}

\section{Introduction}

The theory of bond percolation on $\Z^d$ originated in \cite{BH}. In this model, each edge is declared open, independently of the others, with probability $p$, and closed otherwise. The primary question is if there is an infinite open connected component with positive probability for a given value of $p$, in which case we say that  {\it percolation} occurs.  The existence of a non-trivial phase transition for $d\geq 2$ was established in this seminal article with a {\it critical probability} $p_c(\Z^d)\in (0,1)$, such that percolation occurs for $p>p_c(\Z^d)$ and does not occur for $p<p_c(\Z^d)$.

It is well known that for $d=2$, $p_c(\Z^d) = 1/2$ (see~\cite{K3}). Although the precise value of $p_c(\Z^d)$ is unknown when $d \geq 3$, the asymptotic behavior $2dp_c(\Z^d)=1+o(1/d)$ was obtained independently in \cite{K} and \cite{G}.
Since $\Z^d$ is locally a $(2d)$-ary tree, the model can be locally seen as a Galton-Watson process with i.i.d.~branches, and 
one could say that the critical probabilities of the two models are  
asymptotically equal as $d\rightarrow \infty$.

In this paper we will consider inhomogeneous non-oriented Bernoulli  percolation, also known as {\it anisotropic percolation}, where edges in each direction have distinct parameters. The existence of a non-trivial critical hypersurface is well established and some of its features are known; it is continuous and strictly monotonic in each parameter (see \cite{CLS}) and its behavior near the boundary (i.e.~when one or more parameters are zero) is related to the so-called {\it dimensional crossover} phenomenon (see for instance \cite{GSS}). For $d=2$, the phase-diagram was completely described by Kesten (see page 54 of~\cite{K2}) who proved that percolation occurs if and only if the sum of the two parameters is greater than one or at least one parameter equals one. For $d\geq 3$, a precise determination of  the critical hypersurface seems hopeless so we will focus on how its asymptotic behavior is close to the inhomogeneous Galton-Watson process on the $(2d)-$ary tree.   
We will prove that if the mean number of open edges incident to the origin (i.e. twice the sum of the parameters) is greater than one, percolation occurs under some regularity conditions on the parameters.  Observe that on the inhomogeneous Galton-Watson model, any regularity condition is unnecessary.  Moreover, we give a sufficient condition on the sum of the parameters to guarantee that percolation occurs. The main tool we  use is a monotonic coupling between anisotropic percolation in different dimensions with distinct sets of parameters.

The remainder of the text is organized as follows. In Section \ref{sec:model} we define more precisely the model and state our results. In Section \ref{sec:couplings} we establish the dynamical coupling, and in Section \ref{sec:proofs} we prove the theorems.

\section{The model and main results}\label{sec:model}

We now briefly define the model. Let $\{e_1, \dots, e_d\}$ be the set of canonical vectors of $\Z^d$. For each $i=1, \dots, d$, let $E_i = \{ \langle x, x \pm e_i \rangle: x \in \Z^d\}$ be the set of edges parallel to $e_i$. We denote the edge set by $E(\Z^d):= \cup_i E_i$.

Given  $p_1, \dots, p_d \in [0,1]$, consider a family of independent random variables $\{X_e\}_{e \in E(\Z^d)}$, where, for each $e \in E_i$, $X_e$ has a Bernoulli($p_i$) distribution, $i=1, \dots, n$. Let $\mu_e$ be the law of $X_e$ and  $\Pp = \prod_{e \in E(\Z^d)} \mu_e$ the resulting product measure. We declare an edge $e$ to be {\it open} if $X_e=1$ and {\it closed} otherwise. The model is said to be {\it homogeneous} whenever all the parameters $p_i$ are equal, and {\it inhomogeneous} otherwise.

We denote by $\{x \leftrightarrow y\}$ the event where  $x, y \in \Z^d$ are connected by an open path, i.e., there exist $x_0, \dots, x_n$ such that $x_0 = x$, $x_n =y$ and each $\langle x_{j-1}, x_j \rangle$ belongs to $ E(\Z^d)$ and is open for $j=1, \dots, n$. Let  $\mathcal{C}^d_0 = \{ x \in \Z^d:  0 \leftrightarrow x\}$ be the open cluster of the origin, and
 $|\mathcal{C}^d_0|$ its size.  We define 
\begin{equation*}
    \theta_d(p_1, \dots, p_d) := \Pp(|\mathcal{C}^d_0|= \infty).
\end{equation*}

In what follows, $p_c(\Z^d)= \displaystyle \sup\{p \geq 0 :\theta_d(p,\dots, p)=0\}$  denotes the critical probability for the non-oriented homogeneous model.

Our main result is the following:
\bigskip
\begin{theorem} \label{theo:anisotropic}
Consider inhomogeneous non-oriented Bernoulli bond percolation on $\Z^d$ with parameters $p_1,\dots, p_d \in [0,1)$. There exists a constant $C>0$, independent of the dimension $d \geq 2$, such that, if the following conditions are satisfied 

\quad {\bf C1)} $\delta=\delta(p_1,\dots,p_d):=p_1+\cdots +p_d-1/2 > 0$, and

\quad {\bf C2)} $\displaystyle \max_{1 \leq i \leq d} p_i \leq C \delta^2$,

then $\theta_d(p_1, \dots, p_d)>0$.
\end{theorem}

\smallskip

\begin{remark}\label{r1} At first sight, Condition~{\bf C2} may seem counter-intuitive. To see that some regularity is needed, one can take $p_3 = \cdots = p_{d}=0$ and end up with inhomogeneous non-oriented bond percolation on $\Z^2$. In this case, Kesten (see page 54 in \cite{K2}) proved that the critical curve is $p_1+p_2=1$, hence percolation cannot occur for any $\delta<1/2$.

\end{remark}

\begin{remark}
Condition {\bf C2} may not be optimal, but certainly some condition on the maximal value of the parameters is needed. For instance, let $p={1}/{2d}+{1}/{4d^2}$. We know that (see Formula \eqref{eq: cotapc} below) there exists $d_0$ large enough such that $p<p_c(d)$ for every $d\geq d_0$. In this case, considering  homogeneous percolation on $\Z^d$, $p_1 = \cdots = p_d = p$, we have $\delta = \delta(p, \dots, p) = dp - 1/2 = 1/4d$. Thus, given a constant $C'>0$ and $\alpha >0$, for $d$ sufficiently large, we have $p < C'\delta^{1-\alpha}$ without percolation. Therefore, Condition {\bf C2} could not be replaced by $\max_ip_i< C'\delta^{1-\alpha}$, for any constant $C'>0$ and any $\alpha >0$.
\end{remark}

\begin{remark}
An analogous result was proved in \cite{GPS} for inhomogeneous {\it oriented} Bernoulli bond  percolation. In that case, a martingale approach was used to  weaken Condition \textbf{C2} to $\max_{1 \leq i \leq d} p_i \leq C'\delta$, for some constant $C'$. In fact, if we apply the coupling approach presented in the proof of Theorem \ref{theo:anisotropic} to the oriented case, we  obtain a bound $C\delta^{3/2}$ in Condition \textbf{C2}, which would be worse than the aforementioned result. We also observe that the methods in \cite{GPS} are valid only for $d\geq 4$. However, since the constant $C$ from our present Theorem \ref{theo:anisotropic} is not explicit, there is no gain in extending our result for the oriented case in dimensions $d=2,3$.
\end{remark}

\begin{remark}
The condition $p_1 + \dots + p_d > 1/2$ states that the expected number of open edges incident to each vertex is greater than one, which is analogous to the sufficient condition for the inhomogeneous Galton-Watson process on the $(2d)-$ary tree to survive with positive probability. Observe that whenever $\delta(p_1,\dots,p_d)\leq 0$, that is $p_1 + \cdots + p_d \leq 1/2$, by comparison with a subcritical or critical Galton-Watson process, we have $\theta_d(p_1,\dots,p_d)=0$. 
\end{remark}

\begin{remark}
Theorem~\ref{theo:anisotropic} shows that the asymptotic behaviour of the critical hypersurface is, in some sense,  close to the inhomogeneous Galton-Watson process on the $(2d)-$ary tree. In fact, given $\varepsilon > 0$, if $p_1 + \cdots + p_d > 1/2 + \varepsilon$ and $\max_i p_i \leq C\varepsilon^2$, percolation occurs. Observe that as $\varepsilon$ goes to zero it can only be satisfied for sufficiently large dimensions but, since the value of the constant $C$ is unknown,  Theorem~\ref{theo:anisotropic} does not give an explicit lower bound on the dimension $d$. 
\end{remark}

 The next theorem states that there is a way to get rid of any regularity conditions. More precisely, for values of $\delta$ greater than $ 3\log2 - 1/2$, Theorem \ref{theo:anisotropic} applies without  Condition~\textbf{C2}.

\smallskip

\begin{theorem}\label{prop:3log2}
Consider inhomogeneous non-oriented Bernoulli bond percolation on $\Z^d$ with parameters $p_1,\dots, p_d \in [0,1)$. For any $d \geq 2$, if  
\begin{equation} \label{eq:par3log2}
p_1 + \cdots + p_d > 3 \log 2,
\end{equation}
then $\theta_d(p_1, \dots, p_d) > 0$.
\end{theorem}

\smallskip

\begin{remark}
Recall that for $d=2$, the critical curve is $p_1+p_2=1$ (see Remark \ref{r1}) and simulations in low dimensions suggest that the critical hypersurface is convex. If the latter were the case for all dimensions, Theorem  \ref{prop:3log2} could be stated with constant $1$ instead of $3\log 2$. It would be nice to have such a result. 
\end{remark}

\section{The Dynamical Couplings}\label{sec:couplings}

The first step in proving Theorem \ref{theo:anisotropic} is to construct a monotonic coupling between inhomogeneous percolation on $\Z^d$ with parameters $(p_1, \dots, p_{d-1},\tilde{p_d}$), where $\tilde{p_d}=1-(1-p_d)(1-p_{d+1})$ and  inhomogeneous percolation on $\Z^{d+1}$ with parameters $(p_1, \dots, p_{d-1},p_d, p_{d+1} )$. Monotonic couplings between different percolation processes have been used before (see for instance \cite{GS}).

\begin{proposition} \label{prop:coupling}
Consider inhomogeneous non-oriented Bernoulli bond percolation on $\Z^d$ and on $\Z^{d+1}$. Let $p_1, \dots, p_{d+1} \in [0,1)$ and let $\Tilde{p}_d \in [0,1)$ be such that $$(1-\Tilde{p}_d) = (1-p_d)(1-p_{d+1}).$$ 
Then
    $\theta_{d+1}(p_1, \dots,p_d,  p_{d+1})  \geq \theta_{d}(p_1, \dots, p_{d-1}, \Tilde{p}_{d}) .$
\end{proposition}

Let us start with a description of the proof.
We will construct a dynamical coupling between the percolation process on $\Z^{d+1}$ with parameters $p_1, \dots, p_{d+1}$  and an infection process over $\Z^d$.  We will do it in such a way that the law of infected sites in $\Z^d$ is the same as the law of the open cluster of the origin for anisotropic percolation on $\Z^d$ with parameters $p_1, \dots, p_{d-1},\tilde{p}_d$ and also that, if the infection process survives, the open cluster of the origin of the process in $\Z^{d+1}$ must be infinite. 

To avoid any ambiguity, for each $u \in \{e_1, \dots, e_d\} \in \Z^d$ let $\tilde{u} \in \Z^{d+1}$ denote the same vector $u$ embedded in $\Z^{d+1}$, that is, the vector in $\Z^{d+1}$ where all the first $d$ coordinates are the same as those of $u \in \Z^d$ and the $(d + 1)-$th coordinate is zero.

Before we proceed with a formal proof, we give some words of explanation. The coupling will be  based on a susceptible-infected strategy described as follows. 
We declare  the origin of $\Z^d$ to be the \textit{initial} infected component. Next, at each time-step, we possibly grow the infected component through an available edge to be explored. Precise definitions will be given throughout the proof. On the steps in which we have a new vertex $z \in \Z^d$ added to the infected component, we associate the new vertex $z$ to a vertex $x(z)$ in the open cluster of the origin in $\Z^{d+1}$. 
The function $x$ will be  defined by induction during the proof of Proposition~\ref{prop:coupling} . More precisely, consider a time-step $n$ and consider a vertex $v$ in the infected component of $\Z^d$ and a neighbor $z=v+u$  (where $u \in \{\pm e_1, \dots, \pm e_d\}$) that is not in the infected component at time $n$. Since the vertex $v$ is already in the infected component in $\Z^d$, it was previously associated with some vertex $x(v)$ in the open cluster of the origin in $\Z^{d+1}$. If $\langle x(v), x(v)+\tilde{u} \rangle$ is open, we declare $z$ to be infected and write $x(z) = x(v) + \tilde{u}$. If it happens that $u \in \{\pm e_{d}\}$ and $\langle x(v), x(v)+\tilde{u} \rangle$ is closed, we give a second chance to declare $z$ infected. For $u=\pm e_d$, in this second chance, we declare $z$ to be infected in $\Z^d$ if $\langle x(v), x(v)\pm e_{d+1} \rangle$ is open in $\Z^{d+1}$. In the last case we set $x(z) = x(v)+e_{d+1}$, or $x(z) = x(v) - e_{d+1}$, depending on whether $u= e_d$ or $u=-e_d$.

\begin{proof} 
First, we give a precise description of the susceptible-infected type algorithm. 
We will inductively construct a sequence of sets $(I_n, x_n(I_n), A_n, B_n)_{n \geq 0}$. In this sequence, $I_n$ represents the \textit{infected vertices} in $\Z^d$ up to time $n$ and $x_n(I_n) = \{x_n(v) : v \in I_n\}$ represents the vertices in $\Z^{d+1}$ associated with the infected vertices up to time $n$. As will be clear after the conclusion of the description, at each step of the algorithm we will, if available, explore an edge of $E(\Z^{d})$. In the sequence of sets mentioned above, $B_n  \subset E(\Z^d)$ will represent the \textit{explored edges} up to time $n$ (i.e.~the edges already explored during one step up to time $n$). Finally, $A_n \subset E(\Z^d)$ represents the \textit{available edges} to be explored at time $n$, defined  as follows. Given $I_n$ and $B_n$, let
\begin{equation*}
    A_{n} := \{ \langle v, u \rangle : v \in I_{n} \: \mbox{and} \: u \notin I_{n} \} \cap B_{n}^c.
\end{equation*}
In words, $A_n$ is the set of edges not explored by time $n$ and that are composed of an infected vertex and a non-infected vertex at time $n$.

The dynamics of the process are as follows.
We start with the following settings at time 0:
\begin{itemize}
    \item $I_0 = \{0\} \subset \Z^d$,
    \item $x_0(0) = 0 \in \Z^{d+1}$,
    \item $B_0 = \emptyset \subset E(\Z^d)$ , 
    \item $A_0 := \{ \langle v, u \rangle : v \in I_{0} \: \mbox{and} \: u \notin I_{0} \} \cap B_{0}^c \:\: = \{ \langle 0, \pm e_1\rangle, \dots, \langle 0,\pm e_d\rangle \} \subset E(\Z^d)$.
\end{itemize}

This means that, at time $n=0$, only the vertex $0$ is infected, and it can potentially infect any of its neighbours in the following step, so all edges with the origin as an end-vertex are available.  After that, in each step an infected vertex may or may not propagate the infection to a non-infected vertex through an available edge (if the latter exists). We remark that the occurrence of such propagation will be determined according to the status of some edges of the percolation process on $\Z^{d+1}$. The precise description of such edges and status to be considered will be given in the remainder of the proof.

Given $n \geq 0$, suppose that $I_n$, $x_n:I_n \to \Z^{d+1}$, $B_n$ and $A_n$ are already defined. In case there is no available edge, i.e. $A_n=\emptyset$, the process stops and we set, for all $k\geq 1$,
\begin{itemize}
    \item $ I_{n+k} = I_n,$
    \item $x_{n+k}: I_{n+k} \to \Z^{d+1},$ with $x_{n+k}(v) = x_{n}(v), ~\forall v \in I_{n+k},$
    \item $A_{n+k} = A_n.$
     \item $B_{n+k} = B_n,$
\end{itemize} 
Otherwise, if there exists at least one available edge, i.e. $A_n\not=\emptyset$, let $g_n$ be the earliest edge in $A_n$ according to a fixed ordering. We will now \textit{explore} $g_n$. The precise meaning of the expression `explore' is to execute the next step of the algorithm, in which we analyze the state of edges in the percolation model on $\Z^{d+1}$.  Since $g_n \in A_n$, it must be of the form $\langle v_n, v_n+u_n \rangle$, where $v_n \in I_n$, $u_n \in \{ \pm e_1, \dots, \pm e_{d}\} \subset \Z^d$ and $v_n+u_n \notin I_n$. 
Thus we have two options: either $u_n \in \{ \pm e_1, \dots, \pm e_{d-1}\}$, or $u_n \in \{\pm e_d\}$. At this point, we draw the reader's attention to the notation previously introduced; in what follows, $u_n$ denotes a unitary vector of $\Z^d$, while $\tilde{u}_n$ denotes the same vector embedded in $\Z^{d+1}$, that is, the vector in $\Z^{d+1}$ where all the first $d$ coordinates are the same as those of $u_n \in \Z^d$ and the $(d+1)$-th coordinate is zero. 

Let us treat the first case above. Suppose $u_n \in \{ \pm e_1, \dots, \pm e_{d-1}\}$. Then $v_n$ \textit{infects} $v_n+u_n$ in $\Z^d$ if $\langle x_n(v_n), x_n(v_n)+\tilde{u}_n \rangle$ is open in $\Z^{d+1}$. More precisely, if $\langle x_n(v_n), x_n(v_n)+\tilde{u}_n \rangle$ is open in $\Z^{d+1}$ then we set
\[ I_{n+1} := I_n \cup \{v_n+u_n\},\]
and define $x_{n+1}: I_{n+1} \to \Z^{d+1}$ as
\begin{equation}\label{eq: xn+1a}
    x_{n+1}(v_n+u_n) := x_n(v_n) + \tilde{u}_n \text{ and } ~x_{n+1}(v) = x_n(v) ~ \forall v \in I_n.
\end{equation} 
Otherwise, if $\langle x_n(v_n), x_n(v_n)+\tilde{u}_n \rangle$ is closed in $\Z^{d+1}$, we set $I_{n+1}:= I_n$ and $x_{n+1}: I_{n+1} \to \Z^{d+1}$ as $x_{n+1}(v) = x_{n}(v), ~\forall v \in I_{n+1}$.

In case  $u_n = e_d$, $v_n$ has two chances of infecting $v_n+u_n$, that is,  either  $\langle x_n(v_n), x_n(v_n)+\tilde{u}_n \rangle$ is open in $\Z^{d+1}$ and we set
\begin{equation}\label{eq: x2}
    x_{n+1}(v_n+u_n) := x_n(v_n) + \tilde{u}_n,
\end{equation} 
or $\langle x_n(v_n), x_n(v_n)+\tilde{u} \rangle$ is closed in $\Z^{d+1}$ and $\langle x_n(v_n), x_n(v_n)+e_{d+1} \rangle$ is open in $\Z^{d+1}$, and we write
\begin{equation}\label{eq: x3}
    x_{n+1}(v_n+u_n) := x_n(v_n) + e_{d+1}.
\end{equation}
In both cases, we set
\[ I_{n+1}:= I_n \cup \{v_n+u_n\}\] and conclude the definition of $x_{n+1} : I_{n+1} \to \Z^{d+1}$ with $x_{n+1}(v) = x_{n}(v), ~\forall v \in I_{n}$.
On the other hand, if $\langle x_n(v_n), x_n(v_n)+\tilde{u}_n \rangle$ and $\langle x_n(v_n), x_n(v_n)+e_{d+1} \rangle$ are both closed in $\Z^{d+1}$, we write
$I_{n+1}:= I_n$ and set $x_{n+1}: I_{n+1} \to \Z^{d+1}$ by $x_{n+1}(v) = x_{n}(v), ~\forall v \in I_{n+1}$.

We proceed analogously when $u_n = -e_d$. 

Now that we have \textit{explored} $g_n$, we write 
\[ B_{n+1} := B_n \cup \{ g_n \}.\]
To conclude our induction step, we set
\begin{equation*}
  A_{n+1} := \{ \langle v, u \rangle : v \in I_{n+1} \: \mbox{and} \: u \notin I_{n+1} \} \cap B_{n+1}^c.
\end{equation*}

With the induction step of the algorithm described above, the sequence $(I_n, x_n, A_n,B_n)_{n \geq 0}$ is now well defined. Observe that, by construction, if $v \in I_k \cap I_{m}$ for some $k \neq m$, it yields that $x_k(v) = x_m(v)$. In that way, we can define the following function $x: \cup_j I_j \to \Z^{d+1}$, where $x(v) = x_j(v), ~\forall v \in I_j, ~j \geq 0$.

Observe that the function $x: \cup_j I_j \to \Z^{d+1}$ is injective. 
In fact, given $n \geq 0$, let $w= (w_1, \dots, w_d) \in I_n$, we claim that it satisfies
\begin{itemize}
    \item $x(w)_i=w_i$ for $i= 1, \dots, d-1$,
    \item $x(w)_d+x(w)_{d+1} = w_d$.
\end{itemize}

It can be proved by induction. Since $x_0(0) = 0$, the claim is true for $n = 0$. Given $n \geq 0$, assume that the claim is true for all $w \in I_n$. In the case where $I_{n+1} = I_n$, the claim is then true for all $w \in I_{n+1}$. Consider now the case where $I_{n+1} = I_n \cup \{v_n+u_n\}$ (here we are using the same notation of the algorithm described above). To conclude the proof of the claim, it is sufficient to prove that the claim is true for the new vertex $v_n+u_n$. Observe that, according to the possible definitions for $x(v_n+u_n)$ given in \eqref{eq: xn+1a}, \eqref{eq: x2} and \eqref{eq: x3}, we have either $x(v_n + u_n) = x(v_n) + \tilde{u}_n$ where $u_n = \pm e_i$ for some $i\in \{1, \dots, d\}$, or $x(v_n+u_n) = x(v_n) \pm e_{d+1}$ with $u_n = \pm e_d$. By induction hypothesis, the claim is true for $v_n$. In the first case mentioned above, we add or subtract one in the $i-$th coordinate of the vectors $v_n \in \Z^d$ and $x(v_n) \in \Z^{d+1}$, so the property of the claim extends to $v_n + u_n$. In the second case, we  add or subtract one in the $d-$th coordinate of the vector of $v_n \in \Z^d$ and in the $(d+1)-$th coordinate of the vector $x(v_n) \in \Z^{d+1}$; according to the second property of the claim, it also extends to $v_n+u_n$.

Note that the image of $x$ is contained in the  open cluster of the origin of $\Z^{d+1}$. Since $x$ is injective, if $|\cup_n I_n|$ is infinite, then the open cluster of the origin of $\Z^{d+1}$ must be infinite.  Also, note that $\cup_{n} I_n$ has the same  law as  $\mathcal{C}^{d}_0$, where $\mathcal{C}^{d}_0$ is the  open cluster of the percolation process on  $\Z^d$ with parameters $p_1, \dots, p_{d-1}, 1 - (1-p_d)(1-p_{d+1})$.  Therefore,
\begin{eqnarray*}
 \theta_d(p_1, \dots, p_{d-1}, \tilde{p}_d) &=&  \theta_d\big(p_1, \dots, p_{d-1}, 1 - (1-p_d)(1-p_{d+1})\big)\\  
 &\leq& \theta_{d+1} (p_1, \dots, p_{d+1}),
\end{eqnarray*}
and the proof of Proposition \ref{theo:anisotropic} is complete.
\end{proof}

\section{Proof of Theorems \ref{theo:anisotropic} and \ref{prop:3log2}
}\label{sec:proofs}

The proof of Theorem \ref{theo:anisotropic} will be based on successive applications of Proposition \ref{prop:coupling}.

In this section, we will use the increasing function $q:[0,1) \to [0, \infty)$ defined by 
\begin{equation}\label{eq: q}
    q(p)=-\log (1-p).
\end{equation} 
For each $i=1, \dots, d$, we denote $q(p_i) = q_i$, $\tilde{q}_d = q(\tilde{p}_d)$. 
With this notation, the condition of Proposition \ref{prop:coupling},
\begin{equation*}
        (1-p_d)(1-p_{d+1}) = 1-\tilde{p}_d, 
\end{equation*}
simplifies to
\begin{equation}\label{eq: soma q}
         q_d+q_{d+1} = \tilde{q}_d.
\end{equation}

Given $d \geq m\geq 2$ positive integers, we say that  $({\cal D}_1, \dots, {\cal D}_m)$  is a partition of $\{1, \dots, d\}$, whenever
\begin{equation}\label{eq: partition}
    \bigcup_{i=1}^m {\cal D}_i = \{1, \dots, d\} \: \: \mbox{and}  \: \: {\cal D}_i \cap {\cal D}_j = \emptyset, \: \:\forall i \neq j.
\end{equation}

Successive applications of Proposition \ref{prop:coupling}, together with the notation introduced above, give the following corollary of Proposition \ref{prop:coupling}, that we state as a lemma since it will be used in the proofs of Theorem~\ref{theo:anisotropic} and Theorem~\ref{prop:3log2}.

\begin{lemma}\label{lemma}
Consider inhomogeneous non-oriented Bernoulli bond percolation on $\Z^d$ and on $\Z^m$ with $d \geq m \geq 2$. Let $p_1, \dots, p_{d} \in [0,1)$ and let $\Tilde{p}_1, \dots, \Tilde{p}_m \in [0,1)$ be such that
there exists a partition $({\cal D}_1, \dots, {\cal D}_m)$ of $\{1, \dots, d\}$ where $$\sum_{i \in {\cal D}_j} q_i = q(\Tilde{p}_j), ~j =1, \dots, m.$$
Then
    $\theta_{d}(p_1, \dots, p_d) \geq \theta_{m}(\Tilde{p}_1, \dots, \Tilde{p}_{m}).$
\end{lemma}

Now we prove Theorem \ref{prop:3log2}.

\begin{proof}[Proof of Theorem~\ref{prop:3log2}]
First of all, we recall that $p_c(\Z^{2}) = 1/2$ and $ q(1/2) = \log 2 $. The idea is to apply Lemma~\ref{lemma} in order to compare a $d-$dimensional system with a supercritical $2-$dimensional system.
We will consider two cases: $p_i < 1/2$ for all $i =1, \dots, d$ and $p_i \geq 1/2$ for some $i=1, \dots, d$. Observe that, by hypothesis $p_1 + \cdots + p_d \geq 3\log2$, so the first case can only occur for $d > 4$.

If $p_i < 1/2$, for all $i \geq 1$, then $q_i < \log 2$ for all $i \geq 1$. Since $p_1+ \cdots + p_d \geq 3 \log 2$ and $q_i \geq p_i$, it follows that 
\begin{equation} \label{eq:sumq-log2}
    q_1+ \cdots + q_d \geq 3 \log 2. 
\end{equation}
Let
\begin{equation*}
    m := \min \Big \{j : \sum_{i=1}^{j} p_i > \log 2 \Big\}, 
\end{equation*}

\begin{equation*}
    {\cal D}_1 = \{1, \dots, m\} \:\: \mbox{and} \:\: {\cal D}_2 = \{ m+1, \dots, d\}.
\end{equation*}

By the fact that $q_i < \log 2$, we have that $\sum_{i \in {\cal D}_1} q_i \in (\log 2, 2\log 2)$. Therefore using \eqref{eq:sumq-log2} we obtain
\begin{equation*}
    \sum_{i \in {\cal D}_1} q_i > \log 2 ~~ \mbox{and} ~~ \sum_{i \in {\cal D}_2} q_i = \big( q_1 + \cdots + q_d \big) -  \sum_{j \in {\cal D}_1} q_j > 3\log2 - 2\log2 = \log 2.
\end{equation*}
    
By an application of Lemma~\ref{lemma} we obtain that $\theta_d(p_1, \dots, p_d) > \theta_2(p,p)$, for some $p>1/2$ such that $\sum_{i \in {\cal D}_{\ell}} q_i > q(p) > \log 2$, $\ell = 1, 2$. Since $p_c(\Z^{2}) = 1/2$, the expression $\theta_2(p,p)$ is strictly positive.

In the second case, we can assume, without loss of generality, that $p_1 \geq 1/2$. From \eqref{eq:par3log2}, we have $p_2 + \cdots + p_d \geq 3\log 2 - p_1 >  3\log 2 - 1 > \log 2$, hence $q_2 + \cdots + q_d > \log 2$.

Therefore, by Lemma~\ref{lemma}, $\theta_d(p_1, \dots, p_d) > \theta_2(p_1, p)>0$ for some $p > 1/2$ such that $q_2 + \cdots + q_d > q(p) > \log 2$.
\end{proof}

Note that in the case where $p_i \geq 1/2$ for some $i=1, \dots, d$, the proof given above also works for the hypothesis $p_1 + \cdots + p_{d} \geq 1 + \log2$. Then, for $d=3$ and $d=4$, the hypothesis of Theorem ~\ref{prop:3log2} can be weakened to $p_1 + \cdots + p_{d} \geq 1 + \log2$.

By Theorem  \ref{prop:3log2}, it is enough to prove Theorem \ref{theo:anisotropic} with $0 < \delta < \lambda$, where $\lambda = 3\log 2 - 1/2$. Throughout the rest of the text, we will use the notation $q(p_c(\Z^d)) = q_c(\Z^d)$.

\begin{proof}[Proof of Theorem \ref{theo:anisotropic}]

Our strategy will be to partition $\{1,\dots,d\}$  into $m<d$ subsets. We then apply Lemma~\ref{lemma} so that inhomogeneous percolation on $\Z^d$ will dominate a supercritical  homogeneous percolation on $\Z^m$.

Consider the parameters $p_1, \dots, p_d \in [0,1)$ and $\delta = \delta(p_1 + \cdots + p_d) = p_1 + \cdots + p_d - 1/2$, such that $\delta \in (0, \lambda)$. Recall the definition of partition given in \eqref{eq: partition}: If for some $m \in \N$, there exists a partition  $({\cal D}_1, \dots, {\cal D}_m)$ of $\{1, \dots, d\}$, such that 
\begin{equation} \label{eq:goodpartition}
\sum_{i \in {\cal D}_{j}} q_i > q_c(\Z^m), \:\:\: j = 1, \dots, m, 
\end{equation}
then, by Lemma~\ref{lemma}, we obtain $\theta(p_1, \dots, p_d)>0$. Thus, it is sufficient to prove the existence of a partition $({\cal D}_1, \ldots, {\cal D}_m)$ with the property given in  \eqref{eq:goodpartition}. We start by showing a sufficient condition for the existence of such a partition, and then we show that the hypotheses of  Theorem \ref{theo:anisotropic} imply that sufficient condition.

Write $q_{\max}:=\max_{1 \leq i \leq d} q_i$ and  suppose that, for some $m<d$, 
\begin{equation} \label{eqbound:qc}
    q_1 + \cdots + q_d > (q_c(\Z^m) + q_{\max})(m-1) + q_c(\Z^m).
\end{equation}
We claim that there exists a partition with the property given by  \eqref{eq:goodpartition}. Indeed, take $i_0=0$ and, for each $\ell = 1, \dots, m-1$, let
\begin{equation}\label{eq: inf}
    i_{\ell} := \min \left\{j: \sum_{i=i_{\ell-1}+1}^j q_i > q_c(\Z^m) \right\}.
\end{equation}
Let $i_m=d$ and define
\begin{equation*}
    {\cal D}_{\ell} := \{i_{\ell-1}+1, \dots, i_{\ell}\}, \:\: \text{for each }\ell=1, \dots, m. 
\end{equation*}
By construction, the partition $({\cal D}_1, \dots, {\cal D}_m)$ has the desired property and, therefore, we just need to show that
it is well defined.
Indeed, it follows from   (\ref{eq: inf}) that, for each $\ell=1, \dots, m-1,$
\begin{equation*}
    \sum_{i=1}^{i_{\ell}} q_i \leq (q_c(\Z^m) + q_{\max}) \ell,
\end{equation*}
so, by  \eqref{eqbound:qc}, we have
\begin{equation*}
     \sum_{i=i_{\ell}+1}^{d} q_i > (q_c(\Z^m) + q_{\max}) (m - 1 - \ell) + q_c(\Z^m) \geq q_c(\Z^m),
\end{equation*}
which guarantees the existence of $j$  as needed to define $i_{\ell +1}$ as in  \eqref{eq: inf}. Therefore, $i_0, \dots, i_m$ and $({\cal D}_1, \dots, {\cal D}_m)$ are well defined.

To finish the proof, we will show that the conditions of Theorem \ref{theo:anisotropic} imply the existence of some $m$ as in   \eqref{eqbound:qc}. Since, by hypothesis,
\begin{equation*}
    q_1 + \dots + q_d >p_1+\cdots +p_d = \frac{1}{2} + \delta,
\end{equation*} 
it is sufficient to find $m$ such that 
\[ \frac{1}{2} + \delta \geq m( q_c(\Z^m) + q_{\max}),\]
that is, 
\begin{equation}\label{eq:superarcritico}
\frac{1}{2m} + \frac{\delta}{m} \geq  q_c(\Z^m) + q_{\max}.    
\end{equation}

In \cite{HS} it is shown that 
\begin{equation}\label{eq: cotapc}
p_c(\Z^d) = \frac{1}{2d} + \frac{1}{4d^2} + \frac{7}{16d^3} + O \left( \frac{1}{d^4} \right).
\end{equation}
In particular, there exists a positive constant $C_1$ such that 
\begin{equation}\label{eq: nova}
    q_c(\Z^d) \leq \frac{1}{2d} + \frac{C_1}{d^2}, ~\forall d \geq 2.
\end{equation}

Recall that we are considering $\delta \in (0,\lambda)$. Let $m_{\delta}$ be the positive integer defined as
\begin{equation*}
   m_{\delta} := \left\lceil \frac{2C_1}{\delta} \right\rceil. 
\end{equation*}
Observe that, by the definition of $m_{\delta}$, we have
$\delta \geq 2C_1/m_{\delta} $, hence~\eqref{eq: nova} yields that
\begin{equation} \label{eq: nova2}
\frac{1+ \delta}{2m_{\delta}} \geq \frac{1}{2m_{\delta}} + \frac{C_1}{m_{\delta}^2} \geq q_c(\Z^{m_{\delta}}).
\end{equation}
Also observe that there is a sufficiently large constant $C_2 := 2C_1 + \lambda > 0$ such that $m_{\delta} \leq 2C_1/\delta + 1 \leq C_2/\delta$ for all $\delta \in (0, \lambda)$. This is equivalent to 
\begin{equation}\label{eq: nova3}
    \frac{\delta}{2m_\delta} \geq C_3\delta^2, \quad \forall \delta \in (0, \lambda).
\end{equation}
where $C_3 := 1/(2C_2)$.

Summing the two inequalities in~\eqref{eq: nova2} and \eqref{eq: nova3}, for the case where $q_{\max} \leq C_3\delta^2$, we have that the desired inequality given in~\eqref{eq:superarcritico} is satisfied by taking $m = m_{\delta}$.

To conclude the proof, observe that according to \eqref{eq: q}, the ratio $r(p) := q(p)/p$ is increasing in $p$. Hence, there exists a sufficiently small constant $C > 0$ such that
\begin{equation*}
    \max_{1 \leq i \leq d} p_i \leq C\delta^2 \:\: \Longrightarrow  \:\: q_{\max} \leq C_3\delta^2, \quad \forall \delta \in (0,\lambda).
\end{equation*}
Indeed, let $C$ be such that $Cr(C\lambda^2) = C_3$. Thus we have, for each $i = 1, \dots, d$, $p_i \leq C\delta^2$ yields 
\begin{equation*}
    q_i = r(p_i)p_i < r(C\lambda^2)C\delta^2 = C_ 3\delta^2, \quad \forall \delta \in (0,\lambda).
\end{equation*}
\end{proof}

\begin{remark}
Since $C_1$ could be taken close to $1/4$ as long as we take a sufficiently high dimension, we conclude that the constant $C$ could be taken as close to $C(1/4)$ as one wishes. Unfortunately, for a given value of $C$, we do not have an estimate of the least dimension for which the theorem holds.
\end{remark}

\vspace{1cm}

{\bf Acknowledgements:} The authors thank an anonymous referee who helped improve the readability of the paper and also thank Roger Silva for valuable comments on the first version of the manuscript. 

\vspace{0.5cm}
 
{\bf Funding:} P.A.~Gomes has been supported by S{\~a}o Paulo Research Foundation (FAPESP), grant 2020/02636-3 and grant 2017/10555-0.
R.~Sanchis has been partially supported by Conselho Nacional de Desenvolvimento
Cient{\'i}fico e Tecnol{\'o}gico (CNPq), CAPES and by FAPEMIG (APQ-00868-21 and RED-00133-21).

\vspace{0.5cm}

{\bf Data Availability Statement:} Data sharing is not applicable to this article as no datasets were generated or
analysed during the current study.

\vspace{0.5cm}
{\bf Conflict of interest:} The authors have no competing interests to declare that are relevant to the content of this article.

\end{document}